\documentclass[11pt, a4paper]{amsart}

\usepackage{amssymb, amsthm, amsmath}
\usepackage[latin1]{inputenc}
\usepackage{xcolor}
\usepackage{hyperref}

\newtheorem{theorem}{Theorem}[section]
\newtheorem{corollary}{Corollary}[section]
\newtheorem{lemma}[theorem]{Lemma}
\newtheorem{definition}[theorem]{Definition}

\def \B {\mathbb B}
\def \R {\mathbb R}
\def \s {\mathbb S}
\def \Lp {L_p}
\def \S {{\s^{n-2}}}
\def \Rnm {{\R^n_+}}
\newcommand {\LW}[1] {{L^{#1}(\Rnm, \omega)}}
\def \gl {\operatorname{GL}}

\def \eps{\varepsilon}
\def \partialgradient{\tilde \nabla}
\def \fullgradient{\nabla}

\newcommand {\vol}{\operatorname{vol}}

\def \E {\mathcal E}
\def \EE {\mathcal E_p(f, \omega)}
\def \ZZ {Z_p(f)}
\def \GG {\|\partialgradient f\|_\LW{p}}
\def \ft {\frac{\partial f}{\partial t}}
\def \TT {\left\|\ft\right\|_\LW{p}}

\newcommand \Moma[1] {\int_{{(#1)}_+} \omega(y) dy}

\title{Sharp affine weighted $L^p$ Sobolev type inequalities}

\author{J. Haddad}

\author{C. H. Jim\'enez}

\author{M. Montenegro}

\begin{document}

\maketitle

\begin{abstract}
	We establish sharp affine weighted $L^p$ Sobolev type inequalities by using the $L_p$ Busemann-Petty centroid inequality proved by Lutwak, Yang and Zhang \cite{L-Y-Z}.
Our approach consists in combining in a convenient way the latter one with a suitable family of sharp weighted $L^p$ Sobolev type inequalities obtained by Nguyen \cite{Ng} and allows to characterize all extremizers in some cases. The new inequalities don't rely on any euclidean geometric structure.
\end{abstract}

\section{Introduction and statements}

Weighted Sobolev inequalities play a fundamental role in Analysis and Geometry, see for example \cite{DM} and \cite{F}, among many others references. In \cite{BGL}, Bakry, Gentil and Ledoux proved that for any $n \geq 2$ and $a \geq 0$ with $n + a > 2$, there exists a sharp constant $S(n,a)$ such that for any smooth compactly supported function $f$ on $\R^{n-1} \times \R_+ \subset \R^n$, where $\R_+ = (0, + \infty)$, the specially important weighted inequality holds

\begin{equation} \label{eucineq}
\left( \int_{\R^{n-1}} \int_0^\infty |f|^{2_a^*} x_n^a dx \right)^{\frac{1}{2_a^*}} \leq S(n,a) \left( \int_{\R^{n-1}} \int_0^\infty |\nabla f|^2 x_n^a dx \right)^{\frac{1}{2}}\, ,
\end{equation}
where $2_a^* = \frac{2n + 2a}{n-2 + a}$. Moreover, the value of $S(n,a)$ is given by

\[
S(n,a) = \left( \frac{1}{\pi (n + a) (n-2+a)} \right)^\frac{1}{2} \left[ \frac{2 \pi^\frac{1 + a}{2} \Gamma(n + a)}{\Gamma(\frac{1 + a}{2}) \Gamma(\frac{n + a}{2})} \right]^\frac{1}{n + a}\, ,
\]
where $\Gamma$ is the usual {\it Gamma function} defined by $\Gamma(r) = \int_0^\infty t^{r-1} e^{-t} dt$ for $r > 0$.

Bakry, Gentil and Ledoux proved the euclidean sharp inequality (\ref{eucineq}) by using the {\it Curvature-Dimension} condition. More specifically, the authors transported through the stereographic projection the space $\R^{n-1} \times \R_+$ endowed with the measure having density $x_n^a$ with respect to Lebesgue measure on $\R^{n-1} \times \R_+$ to a half sphere $\{ x \in \s^n : x_n > 0\}$ endowed with the measure having density $x_n^a$ with respect to the Riemannian measure of the sphere. The operator associated to this transportation, which is given by $\Delta_{\s^n} + a \nabla ({\rm log} x_n)$, satisfies the {\it Curvature-Dimension} condition $CD(n + a - 1, n + a)$ on the half sphere $\{ x \in \s^n : x_n > 0\}$. The proof then follows by using that this condition implies a sharp Sobolev inequality on the same half sphere which is, module the stereographic projection map, equivalent to (\ref{eucineq}).

The following extension of (\ref{eucineq}) to $L^p$-norms of gradients,

\begin{equation} \label{eucineq-p}
\left( \int_{\R^{n-1}} \int_0^\infty |f|^{p_a^*} x_n^a dx \right)^{\frac{1}{p_a^*}} \leq S(n,p,a) \left( \int_{\R^{n-1}} \int_0^\infty |\nabla f|^p x_n^a dx \right)^{\frac{1}{p}}\, ,
\end{equation}
with $p_a^* = \frac{(n + a)p}{n-p + a}$, was established by Cabré, Ros-Oton and Serra for $n \geq 2$, $a \geq 0$ and $1 < p < n + a$ in the recent papers \cite{CRS1} and \cite{CRS2}. The best constant $S(n,p,a)$ is given by

\[
S(n,p,a) = \left( \frac{(p-1)^{p-1}}{(n + a) (n-p+a)^{p-1}} \right)^\frac{1}{p} \left[ \frac{\Gamma(n + a)}{\Gamma(\frac{(n + a)(p-1)}{p} + 1) \Gamma(\frac{n + a}{p}) \int_{\B_+} x_n^a dx} \right]^\frac{1}{n + a}\, ,
\]
where $\B_+ = \{x \in B_1(0):\; x_n > 0\}$ and

\[
\int_{\B_+} x_n^a dx = \frac{\pi^{\frac{n-2-a}{2}} \Gamma(\frac{1 + a}{2})}{2\Gamma(\frac{n + a + 2}{2})}\, .
\]
Their proof was based on a key weighted isoperimetric inequality on $\R^{n-1} \times \R_+$ which has been obtained in \cite{CRS1} for monominal weights and in \cite{CRS2} for general weights on open convex cones. Note that the latter one clearly includes the half space $\R^{n-1} \times \R_+$ as a special case. The proofs of the weighted isoperimetric inequality are based on the {\it ABP} method applied to the Neumann type operator $\mathcal L u(x) := x_n^{-a} {\rm div}(x_n^a \nabla u(x))$. Then, the authors proved the euclidean $L^p$ version of (\ref{eucineq}) by combining the referred isoperimetric inequality with a weighted radial rearrangement argument of Talenti on $\R^{n-1} \times \R_+$. More recently, Nguyen \cite{Ng} established non-euclidean counterparts of (\ref{eucineq-p}) and also weighted $L^p$ Gagliardo-Nirenberg and log-Sobolev inequalities through a mass transport approach inspired on the famous work of Cordero-Erausquin, Nazaret and Villani \cite{CNV} in which the technique has been introduced for non-euclidean $L^p$ Sobolev and Gagliardo-Nirenberg inequalities on the whole space.

By using an argument of dimension reduction applied to the weighted inequality (\ref{eucineq-p}) in dimension $n + 1$, as done in \cite{BGL}, one derives the classical sharp $L^p$ Sobolev inequality for smooth compactly supported functions on $\R^n$,

\begin{equation} \label{Sobolev}
\left( \int_{\R^n} |f(y)|^{p^*} dy \right)^{\frac{1}{p^*}} \leq S(n,p,0) \left( \int_{\R^n}|\nabla f(y)|^p dy \right)^{\frac{1}{p}}
\end{equation}
where $p^* = p_0^*$, whose related literature is extremely rich. As is well known, Inequality (\ref{Sobolev}) was established by Federer and Fleming \cite{FF}, Fleming and Rishel \cite{FR} and Maz'ja \cite{Ma} for $p = 1$ and Aubin \cite{Au} and Talenti \cite{Ta} for $1 < p < n$. In addition, the sharp $L^1$ Sobolev inequality is equivalent to the classical isoperimetric inequality in euclidean $n$-space which is also the geometric core of the sharp $L^p$ Sobolev inequality for $1 < p < n$. Years later, Lutwak, Yang and Zhang introduced and proved a stronger sharp $L^p$ Sobolev inequality which imply its classical counterpart. Precisely, the sharp {\it affine} $L^p$ Sobolev inequality states for any smooth compactly supported function on $\R^n$ that

\begin{equation} \label{affineSobolev}
\left( \int_{\R^n} |f(y)|^{p^*} dx \right)^{\frac{1}{p^*}} \leq S(n,p,0) \mathcal{E}_p(f)\, ,
\end{equation}
where

\[
\mathcal{E}_p(f) := c_{n,p} \left(\int_{\s^{n-1}}\|\fullgradient_\xi f\|_{L^p(\R^n)}^{-n}d\xi\right)^{-\frac 1n}
\]
with
\begin{eqnarray*}
c_{n,p} = \left(n \rho_n\right)^{\frac 1n} \left(\frac{n \rho_n \rho_{p-1}}{2 \rho_{n+p-2}}\right)^{\frac 1p},
\end{eqnarray*}
where $\fullgradient_\xi f(y) = \fullgradient f(y) \cdot \xi$ and $\rho_k$ denotes the volume of the unit euclidean ball $\B^k$ in $\R^k$ and for $k$ real, one has

\[
\rho_k = \frac{\pi ^{\frac{k}{2}}}{\Gamma \left(\frac{k}{2}+1\right)}\, ,
\]
where $\Gamma(\cdot)$ denotes as usual the Gamma function.

Inequality (\ref{affineSobolev}) was established by Zhang \cite{Z} for $p = 1$ and Lutwak, Yang and Zhang \cite{L-Y-Z-1} for $1 < p < n$. The latter also showed that

\[
\mathcal{E}_p(f) \leq \left( \int_{\R^n}|\nabla f(y)|^p dy \right)^{\frac{1}{p}}\, ,
\]
implying readily that (\ref{affineSobolev}) is stronger than (\ref{Sobolev}). These results are rather surprising, since it was not at all expected that the classical $L^p$ Sobolev inequalities would admit {\it affine} versions that are independent of any euclidean structure. Parallel to the classical case $p=1$, the sharp {\it affine} $L^1$ Sobolev inequality is equivalent to an isoperimetric inequality called Petty projection inequality which was shown by Zhang \cite{Z}. The proof for $1 < p < n$ bases on a family of $L_p$ {\it affine} isoperimetric inequalities established by Lutwak, Yang and Zhang \cite{L-Y-Z}, known as the $L_p$ Petty projection inequalities, and the solution to the $L_p$ Minkowski problem by the same authors (see \cite{L-Y-Z-2}). A central tool obtained in \cite{L-Y-Z} and used in the proof of the $L_p$ Petty projection inequalities is the $L_p$ Busemann-Petty centroid inequality. The latter generalizes the classical Busemann-Petty centroid inequality due to Petty \cite{Petty} (see also \cite{C-G} for an alternative proof) that compares the ratio between the volume of a convex body and that of its centroid body. For other {\it affine} isoperimetric and functional inequalities all of which directly imply their euclidean counterparts, see, e.g., Cianchi, Lutwak, Yang and Zhang \cite{CLYZ}, Lutwak, Yang and Zhang \cite{L-Y-Z-3}, Haberl and Schuster \cite{HaSc1}, \cite{HaSc2}, Haberl, Schuster and Xiao \cite{HSX}, Ludwig, Xiao and Zhang \cite{LXZ} and Wang \cite{W1}, \cite{W2}. Inequality (\ref{affineSobolev}) has been proved by using an alternative method introduced recently by Haddad, Jiménez and Montenegro \cite{HJM} which uses the $L_p$ Busemann-Petty centroid inequality as a fundamental tool and has the advantage of not depending on the solution to the $L_p$ Minkowski problem. The efficiency of method is also illustrated in \cite{HJM} with an alternative proof of the well known sharp {\it affine} $L^p$ Gagliardo-Nirenberg and log-Sobolev inequalities.

The main motivation of the present work is to introduce sharp {\it affine} weighted Sobolev type inequalities on $\R^n_+$. More specifically, affine weighted $L^p$ Sobolev, Gagliardo-Nirenberg and log-Sobolev inequalities with sharp constants, which are significantly stronger than and imply the corresponding classical ones established recently by Nguyen \cite{Ng}. This demands a different approach from those developed in \cite{CRS1}, \cite{CRS2} and \cite{Ng} for the proof of the inequality (\ref{eucineq-p}) on half spaces and in \cite{L-Y-Z} and \cite{Z} for the proof of the {\it affine} inequality (\ref{affineSobolev}) on the whole space. The new inequalities again are independent with respect to the fixed norm in the euclidean $n$-space, that is, they depend only on the vector space structure and Lebesgue measure of $\R^n$, so that they are invariant under all affine transformations of $\R^n_+$. This fact is quite interesting since the inequality (\ref{eucineq-p}) relies heavily on the euclidean geometric structure of $\R^n$.
Our argument, as opposed to those used in proofs of the inequalities mentioned above, does not follow from neither symmetrization and transport arguments nor the $L_p$ Petty projection inequality and the solution to the $L_p$ Minkowski problem. The essential isoperimetric inequality behind our approach is the $L_p$ Busemann-Petty centroid inequality by Lutwak, Yang and Zhang. As pointed out by Lutwak in \cite{Lut}, the close connection between the Petty projection and the Busemann-Petty inequalities makes the use of the latter in this type of approach far from surprising.
The second fundamental tool in our proof is the version of the inequality (\ref{eucineq-p}) and  others obtained by Nguyen \cite{Ng} for general norms of gradients.

Our technique reveals in an explicit and elementary way the geometric nature behind {\it affine} inequalities. A connection between Sobolev type inequalities for general norms and {\it affine} Sobolev type inequalities via the $L_p$ Busemann-Petty centroid inequality can be traced back to \cite{L-Y-Z-3} where the authors studied the equivalence between the $L_p$ Minkowski problem and the norms that minimize the right-hand side of the normed Sobolev inequalities obtained in \cite{CNV}.

In order to state the main theorems, some notations should be introduced.

Let $n \geq 2$ and $p \geq 1$. We shall denote each point of the whole space $\R^n$ or half-space $\Rnm = \R_+ \times \R^{n-1}$ by $(t, x)$ or $y$ and each point of the euclidean sphere $\s^{n-2}$ in $\R^{n-1}$ by $\xi$.

For smooth functions $f(t, x)$ with compact support on $\R^n$, we denote by $\ft$ and $\partialgradient f$ respectively the partial derivative with respect to the variable $t$ and the gradient with respect to the variable $x$. We also stands for $\partialgradient_\xi f$ the directional derivative of $f$ with respect to the second variable in the direction of $\xi \in \s^{n-2}$, namely $\partialgradient_\xi f = \partialgradient f \cdot \xi$.

Let $a \geq 0$ and consider the function $\omega$ on $\Rnm$ defined by $\omega(t, x) = t^a$. For $1 \leq p < n + a$, we consider the weighted space $\LW{p}$ endowed with the following norm
\[
\|f\|_{\LW p} = \left( \int_{\Rnm} |f(t, x)|^p \omega(t, x) dt\ dx \right)^{\frac{1}{p}}\, .
\]

We introduce the new integral term
\begin{eqnarray*}
\EE &=&c_{n-1,p} \left( \int_{\s^{n-2}} \| \partialgradient_\xi f\|_\LW{p}^{1-n} d\xi \right) ^{\frac 1 {1 - n}} \\
&=& c_{n-1,p} \left(\int_{\s^{n-2}} \left( \int_{\R^n_+} |\partialgradient_\xi f(y)|^p \omega(y) dy \right)^{-\frac {n-1}p} d\xi \right)^{-\frac 1{n-1}}\, ,
\end{eqnarray*}
where $c_{n,p}$ was introduced above.

In the sequel, we state the three main theorems of this work. The first one deals with the affine counterpart of the sharp weighted $L^p$ Sobolev inequality (\ref{eucineq-p}).

\begin{theorem}
\label{mainthmSob}
Let $a \geq 0$ and $1 \leq p < n + a$. For any smooth function $f$ with compact support on $\R^n$, we have
\begin{equation}
\label{mainthmSob_ineq}
\|f\|_\LW{p_a^*} \leq \mathcal S_{n,p,a}  \EE^{\frac{n-1}{n+a}} \TT^{\frac{1+a}{n+a}}\, ,
\end{equation}
where $\mathcal S_{n,p,a}$ is sharp for this inequality and its value is computed in the appendix A. Moreover, equality holds in \eqref{mainthmSob_ineq} if

\[
f(t,x) = \left\{
\begin{array}{lll}
c \left(1+ |\lambda t|^\frac{p+1}{p} + |B(x-x_0)|^\frac{p+1}{p} \right)^{-\frac {n + a -p} p }  && {\rm if} \ \ p>1\\
c {\bf 1}_\B( \lambda t, B(x-x_0)) &&{\rm if} \ \ p=1
\end{array}
\right.
\]
for some constants $c \in \R$, $\lambda \neq 0$, $x_0 \in \R^{n-1}$ and $B \in \gl_{n-1}$, where ${\bf 1}_\B$ stands for the characteristic function of the unit ball $\B$ centered at the origin and $\gl_{n-1}$ denotes the set of invertible real $(n-1) \times (n-1)$-matrices. Moreover, the characterization with above extremals holds for $p > 1$.
\end{theorem}

Remark that Inequality \eqref{mainthmSob_ineq} is invariant under affine transformations of $\Rnm$. In precise terms, denote by $\gl_{n,+}$ the set of matrices of the form
\begin{equation}
\label{mainthmSob_invariantmatrix}
\left(
\begin{array}{cccc}
\lambda & 0 & \cdots & 0\\
0       &   &        &  \\
\vdots  &   &   B    &  \\
0       &   &        &
\end{array}
\right)
\end{equation}
where $\lambda > 0$ and $B \in \gl_{n-1}$. Then, Inequality \eqref{mainthmSob_ineq} is $\gl_{n,+}$ invariant. In particular, it does not depend on the
euclidean structure of $\R^n$. Note also that the functions ${\bf 1}_\B( \lambda t, B(x-x_0))$ are characteristic functions of ellipsoids.

As a consequence of Young's inequality we get the following consequence of Theorem \ref{mainthmSob}:

\begin{corollary}
\label{conseqaffSoblev}
Let $a \geq 0$ and $1 \leq p < n + a$. For any smooth function $f$ with compact support on $\R^n$, we have

\begin{equation}
\label{mainthmSob_ineq_2}
\|f\|_\LW{p_a^*} \leq \mathcal K_{n,p,a} \left( \EE^p + \TT^p \right)^{\frac{1}{p}}\, ,
\end{equation}
where $\mathcal K_{n,p,a} = \mathcal S_{n,p,a} (1+a)^{\frac{1+a}{p (n+a)}} (n-1)^{\frac{n-1}{p (n+a)}} (n+a)^{-\frac{1}{p}}$ is sharp for this inequality. Moreover, equality holds in \eqref{mainthmSob_ineq_2} if

\[
f(t,x) = \left\{
\begin{array}{lll}
c \left(1+ |\lambda_B t|^\frac{p+1}{p} + |B(x-x_0)|^\frac{p+1}{p} \right)^{-\frac {n + a -p} p }  && {\rm if} \ \ p>1\\
c {\bf 1}_\B( \lambda_B t, B(x-x_0)) &&{\rm if} \ \ p=1
\end{array}
\right.
\]
for some constant $c \in \R$, $x_0 \in \R^{n-1}$ and $B \in \gl_{n-1}$, where $\lambda_B = \det(B)^{\frac 1{n-1}}$. Moreover, the characterization with above extremals holds for $p > 1$.

\end{corollary}

The sharp affine weighted $L^p$ Gagliardo-Nirenberg inequality states that

\begin{theorem}

\label{mainthmGNS}
Let $a \geq 0$, $1 \leq p < n + a$ and $\alpha \in (0, \frac {n + a}{n + a-p}]$ with $\alpha \neq 1$. For any smooth function $f$ with compact support on $\R^n$, we have:
\begin{enumerate}
\item[(a)]
For $\alpha > 1$,
\begin{equation}
\label{mainthmGNS_ineq1}
 \|f\|_{\LW{\alpha p}} \leq \left( \mathcal G_{n,p,a,\alpha}  \EE^{\frac{n-1}{n+a}} \TT^{\frac{1+a}{n+a}} \right)^{\theta} \|f\|^{1-\theta}_{\LW{\alpha(p-1)+1} }\, ,
\end{equation}
where

\[
\theta = \frac {(n + a)(\alpha-1)}{\alpha(n p + ap -(\alpha p + 1 -\alpha)(n+a - p))} \in (0,1)
\]
and $\mathcal G_{n,p,a,\alpha}$ is sharp for this inequality and its value is computed in the appendix A. Moreover, equality holds in \eqref{mainthmGNS_ineq1} if

\[
f(t,x) = \left\{
\begin{array}{lll}
c ( 1 + |\lambda t|^{\frac p{p-1}}+|B (x-x_0)|^{\frac p{p-1}})^{\frac{1}{1-\alpha}}  && {\rm if} \ \ p>1\\
c {\bf 1}_\B( \lambda t, B(x-x_0)) &&{\rm if} \ \ p=1
\end{array}
\right.
\]
for some $c \in \R$, $\lambda \neq 0$, $x_0 \in \R^{n-1}$ and $B \in \gl_{n-1}$;

\item[(b)]
For $\alpha < 1$,
\begin{equation}
\label{mainthmGNS_ineq2}
\|f\|_{\LW{\alpha(p-1)+1} } \leq \left( \mathcal N_{n,p,a,\alpha} \EE^{\frac{n-1}{n+a}} \TT^{\frac{1+a}{n+a}} \right)^{\theta} \|f\|^{1-\theta}_{\LW{\alpha p}}\, ,
\end{equation}
where

\[
\theta = \frac {(n+a)(1-\alpha)}{(\alpha p + 1 -\alpha)(n - \alpha(n-p))} \in (0,1)
\]
and $\mathcal N_{n,p,a,\alpha}$ is sharp for this inequality and its value is computed in the appendix A. Moreover, equality holds in \eqref{mainthmGNS_ineq2} if

\[
f(t,x) = \left\{
\begin{array}{lll}
c ( 1 - |\lambda t|^{\frac p{p-1}}-|B (x-x_0)|^{\frac p{p-1}})_+^{\frac{1}{1-\alpha}}  && {\rm if} \ \ p>1\\
c {\bf 1}_\B( \lambda t, B(x-x_0)) &&{\rm if} \ \ p=1
\end{array}
\right.
\]
for some $c \in \R$, $\lambda \neq 0$, $x_0 \in \R^{n-1}$ and $B \in \gl_{n-1}$. Here $f_+$ denotes the positive part of $f$.
\end{enumerate}
\end{theorem}

The sharp affine weighted $L^p$ entropy inequality asserts that

\begin{theorem}
\label{mainthmLog}
Let $a \geq 0$ and $p \geq 1$. For any smooth function $f$ with compact support on $\R^n$ such that $||f||_\LW{p} = 1$, we have
\begin{equation}
\label{mainthmLog_ineq}
Ent_\omega(|f|^p) := \int_\Rnm |f(y)|^p \log(|f(y)|^p) \omega(y) dy \leq \frac{n + a}p \log\left[ \mathcal L_{n,p,a}  \left( \EE^{\frac{n-1}{n+a}} \TT^{\frac{1+a}{n+a}}\right)^p \right]\, ,
\end{equation}
where $\mathcal L_{n,p,a}$ is sharp for this inequality and its value is computed in the appendix A. Moreover, equality holds in \eqref{mainthmLog_ineq} if

\[
f(t,x) = \left\{
\begin{array}{lll}
c e^{-|\lambda t|^\frac{p}{p-1} - |B(x-x_0))|^\frac{p}{p-1}}  && {\rm if} \ \ p>1\\
c {\bf 1}_\B( \lambda t, B(x-x_0)) &&{\rm if} \ \ p=1
\end{array}
\right.
\]
for some $c \in \R$, $\lambda \neq 0$, $x_0 \in \R^{n-1}$ and $B \in \gl_{n-1}$ such that $||f||_\LW{p} = 1$.
\end{theorem}

Following similar ideas used in \cite{L-Y-Z-1}, by using H\"older's inequality and Fubini's theorem we easily get the estimate of $\EE$ for any $p \geq 1$:

\begin{equation} \label{stronger}
\EE \leq \left( \int_{\Rnm} |\tilde{\nabla} f(t, x)|^p \omega(t, x) dt\ dx \right)^{\frac{1}{p}}\, .
\end{equation}
This implies that all inequalities stated in theorems are stronger than their euclidean counterparts. Note also that extremal functions are taken on appropriate extended spaces in each statement. The classical version of Theorem \ref{mainthmGNS} was partially established (Part (a)) in the whole space $\R^n$ for $p > 1$ and $a = 0$ by Del Pino and Dolbeault in \cite{DPDo} and in the half space $\R_+^n$ under general norms for $p > 1$ and $a \geq 0$ by Nguyen \cite{Ng}. The case $p = 1$ follows by letting $p \rightarrow 1^+$ in \eqref{mainthmGNS_ineq1} and \eqref{mainthmGNS_ineq2}, noting that the limits $\mathcal G_{n,1,a,\alpha}$ and $\mathcal N_{n,1,a,\alpha}$ are equal to $\mathcal S_{n,1,a}$ and by evoking Hölder's inequality and the weighted inequality (\ref{mainthmSob_ineq}) for $p=1$. Already the classical version of Theorem \ref{mainthmLog} with $a=0$ was proved in $\R^n$ by Ledoux \cite{Le} and Beckner \cite{Be} for $p = 1$, by Beckner \cite{Be} and Del Pino and Dolbeault \cite{DPDo} for $1 < p < n$ and by Gentil \cite{G} for any $p > 1$ and in $\R_+^n$ for $a \geq 0$ and $p \geq 1$ by Nguyen \cite{Ng} considering abstract norms.

The paper is organized as follows. In section \ref{sec_notation} we fix some notations and present some background in convex analysis to be used along of paper. In section \ref{sec_lemmas} we prove two central tools in our method (Lemmas 3.5 and 3.6). These lemmas together with the $L^p$ Busemann-Petty Centroid inequality are just what we need in the proof of Theorem \ref{mainthmSob} which is presented in section \ref{sec_proof}. Moreover, in this section we provide some additional comments on the essence of the approach and characterization of extremal functions. Reasoning in a similar line, we prove Theorems \ref{mainthmGNS} and \ref{mainthmLog} at the end of section \ref{sec_proof}. For a better organization and reader's convenience we include the appendix section A dedicated to the computation of all sharp constants stated in results and the appendix section B which describes how varies each quantity of theorems under a linear change of coordinates.

\section{Background in Convex Analysis}
\label{sec_notation}

We recall that a convex body $K\subset\R^n$ is a convex compact subset of $\R^n$ with non-empty interior.

For $K\subset\mathbb R^n$ as before, its support function $h_K$ is defined as

$$
h_K(y)=\max\{\langle y, z\rangle :\ z\in K\}\, .
$$
The support function, which describes the (signed) distances of supporting hyperplanes to the origin, uniquely characterizes $K$. We also have the gauge $\|\cdot\|_K$ and radial $r_K(\cdot)$ functions of $K$ defined respectively as

\[
\|y\|_K:=\inf\{\lambda>0 :\  y\in \lambda K\}\, ,\quad y\in\R^n\setminus\{0\}\, ,
\]

\[
r_K(y):=\max\{\lambda>0 :\ \lambda y\in K\}\, ,\quad y\in\R^n\setminus\{0\}\, .
\]
Clearly, $\|y\|_K=\frac{1}{r_K(y)}$. We also recall that $\|\cdot\|_K$ it is actually a norm when the convex body $K$ is symmetric (i.e. $K=-K$). On the other hand, any centrally symmetric convex body $K$ is the unit ball for some norm in $\R^n$.

For $K\subset \R^n$ we define its polar body, denoted by $K^\circ$, by

\[
K^\circ:=\{y\in\R^n :\ \langle y,z \rangle\leq 1\quad \forall z\in K\}\, .
\]
Evidently, $h_K=r_{K^{\circ}}$. It is also easy to see that $(\lambda K)^\circ=\frac{1}{\lambda}K^\circ$ for all $\lambda>0$. A simple computation using polar coordinates shows that

\[
\vol(K)=\frac{1}{n}\int_{\s^{n-1}}r_K^n(y)dy=\frac{1}{n}\int_{\s^{n-1}}\|y\|^{-n}_K dy\, .
\]

For a given convex body $K\subset\R^n$ we find in the literature many associated bodies to it, in particular Lutwak and Zhang introduced \cite{L-Z} for a body $K$ its $\Lp$-centroid body denoted by $\Gamma_pK$. This body is defined by

\[
h_{\Gamma_pK}^p(y):=\frac{1}{a_{n,p}\vol(K)}\int_{K}|\langle y,z\rangle|^p dz\quad \mbox{ for }y\in\R^n\, ,
\]
where

\[
a_{n,p} = \frac{\rho_{n+p}}{\rho_2 \rho_n \rho_{p-1}}\, .
\]

There are some other normalizations of the $\Lp$-centroid body in the literature, the previous one is made so that $\Gamma_p \B^n=\B^n$ for the unit ball in $\R^n$ centered at the origin.

Inequalities (usually affine invariant) that compare the volume of a convex body $K$ and that of an associated body are common in the literature. For the specific case of $K$ and $\Gamma_pK$, Lutwak, Yang and Zhang \cite{L-Y-Z} (see also \cite{C-G} for an alternative proof) came up with what it is known as the $\Lp$ Busemann-Petty centroid inequality, namely
\begin{equation}
\label{not_BPineq}
\vol(\Gamma_pK)\geq \vol(K)\, .
\end{equation}
This inequality is sharp if, and only if, $K$ is a $0$-symmetric ellipsoid. For a comprehensive survey on $\Lp$ Brunn-Minkowski theory and other topics within convex geometry we refer to \cite{Sch} and references therein.

\section{Fundamental lemmas}
\label{sec_lemmas}
Let $C:\R^n\rightarrow \R$ be an even convex function such that $C(x) > 0$ for all $x \neq 0$. Assume also that $C$ is positively $q$-homogeneous with $q>1$, that is

\[
C(\lambda x)=\lambda^q C(x),\quad \forall \lambda\geq 0,\ x\in\R^n\, .
\]
Denote by $C^*$ its Legendre transform defined by

\[
C^*(y) = \sup_{z \in \R^n} \{ \langle y, z \rangle - C(z) \}\, .
\]
One knows that $C^*$ is also even, convex, $p$-homogeneous and $C^*(y) > 0$ for all $y \neq 0$, where $p > 1$ satisfies $\frac{1}{p}+\frac{1}{q}=1$. Equivalently, the assumptions on $C$ can be resumed by saying simply that $C$ is a power $q$ of an arbitrary norm on $\R^n$.

Let $C$ be a function as before. Denote

\[
K_C = \{y \in \R^n :\ C(y) \leq 1 \}\, .
\]
It is easy to see that $K_C$ is a centrally symmetric convex body with non-empty interior in $\R^n$. Moreover, $K_C$ is defined by the norm $\|y\|_{K_C}=C(y)^{\frac{1}{q}}$.

The starting point of our work consists in the following three theorems proved in \cite{Na}. They are the non-euclidean weighted versions of the Sobolev, Gagliardo-Nirenberg and entropy inequalities, respectively. As quoted in the introduction, Gagliardo-Nirenberg inequalities established by Nguyen occur for $p = 1$.

\begin{theorem}
\label{nguyenthmSob}
Let $C$ be an even convex positively $q$-homogeneous function with $q > 1$ such that $C(x) > 0$ for all $x \neq 0$ and let $p > 1$ with $\frac{1}{p} + \frac{1}{q} = 1$. Assume $a \geq 0$ and $p< n+a$. For any smooth function $f$ with compact support on $\R^n$, we have

\begin{equation}
\label{nguyenthmSob_ineq}
\|f\|_\LW{p_a^*} \leq S(n,a,p) \left( \Moma{(K_C)} \right)^{-\frac{1}{n + a}}  \left( \int_\Rnm C^*(\fullgradient f(y)) \omega(y) dy \right)^{\frac{1}{p}}\, ,
\end{equation}
where $p_a^* = \frac{(n + a)p}{n + a-p}$, $(K_C)_+ = K_C \cap \R^n_+$ and

\[
S(n,a,p) = {p^{\frac{1}{p}} q^{\frac{1}{q}} \left(\frac{(p-1)^{p-1}}{(n + a) \left(n + a-p\right)^{p-1}}\right)^{\frac{1}{p}} \left({\frac{\Gamma \left(\frac{n + a}{p}\right) \Gamma \left(\frac{(n + a) (p-1)}{p}+1\right)}{\Gamma \left(n + a\right)}}\right)^{-\frac{1}{n + a}}}
\]
is sharp for this inequality. Moreover, equality holds in \eqref{nguyenthmSob_ineq} if, and only if,

\[
f(t,x) = c h_{p,a}(\lambda t, \lambda (x - x_0))
\]
for some $c \in \R, \lambda > 0$ and $x_0 \in \R^{n-1}$, where $h_{p,a}$ is given by

\[
h_{p,a}(t, x) = \left(1 + C(t, x)\right)^{-\frac {n + a - p} p }.
\]
\end{theorem}

\begin{theorem}
\label{nguyenthmGNS}
Let $C$ be an even convex positively $q$-homogeneous function with $q > 1$ such that $C(x) > 0$ for all $x \neq 0$ and let $p > 1$ with $\frac{1}{p} + \frac{1}{q} = 1$. Assume $a \geq 0$, $p < n+a$ and $\alpha \in (0, \frac{n + a}{n + a - p}]$ with $\alpha \neq 1$. For any smooth function $f$ with compact support on $\R^n$, we have:

\begin{enumerate}
\item[(i)] If $\alpha > 1$,

\begin{equation}
\label{nguyenthmGNS_ineq1}
\|f\|_{\LW{\alpha p}} \leq G_{n,a}(\alpha, p)\left(\Moma{(K_C)} \right)^{-\frac \theta {n+a}} \left( \int_\Rnm C^*(\fullgradient f(y)) \omega(y) dy \right)^{\frac\theta p} \|f\|^{1-\theta}_{\LW{\alpha(p-1)+1} }\, ,
\end{equation}
where

\[
\theta = \frac {(n + a)(\alpha-1)}{\alpha(n p + ap -(\alpha p + 1 -\alpha)(n+a - p))} \in (0,1),
\]
and
\[
G_{n,a}(\alpha, p) = \left[\frac{\beta(\alpha-1)^p}{(n+a) q^{p-1}} \right]^{\frac \theta p} \left[\frac{q \beta - n - a}{q \beta}\right]^{\frac 1 {\alpha p}} \left[\frac{\Gamma(\beta)}{\Gamma(\beta-\frac{n+a}{q}) \Gamma(\frac{n+a}{q}+1)}\right]^{\frac \theta {n+a}}
\]
is sharp for this inequality, where $\beta = \frac {\alpha (p-1)+1}{\alpha-1}$.

\item[(ii)] If $\alpha < 1$,

\begin{equation}
\label{nguyenthmGNS_ineq2}
\|f\|_{\LW{\alpha(p-1)+1}} \leq N_{n,a}(\alpha, p) \left( \Moma{(K_C)} \right)^{-\frac \theta {n+a}} \left( \int_\Rnm C^*(\fullgradient f(y)) \omega(y) dy \right)^{\frac\theta p} \|f\|^{1-\theta}_{\LW{\alpha p} }\, ,
\end{equation}
where
\[
\theta = \frac {(n+a)(1-\alpha)}{(\alpha p + 1 -\alpha)(n - \alpha(n-p))} \in (0,1)
\]
and

\[
N_{n,a}(\alpha, p) = \left[\frac{\gamma(1-\alpha)^p}{(n+a) q^{p-1}}\right]^{\frac \theta p} \left[\frac{q \gamma}{q \gamma + n+a}\right]^{\frac {1-\theta} {\alpha p}} \left[ \frac{\Gamma(\gamma+1+\frac{n+a}q)}{\Gamma(\gamma+1) \Gamma(\frac{n+a}q + 1) } \right]^{\frac \theta {n+a}}
\]
is sharp for this inequality, where $\gamma = \frac {\alpha (p-1) +1}{1-\alpha}$.
\end{enumerate}
Moreover, equality in \eqref{nguyenthmGNS_ineq1} and \eqref{nguyenthmGNS_ineq2} holds if

\[
f(t,x) = c h_\alpha(\lambda t, \lambda(x-x_0) )
\]
for some $c \in \R, \lambda > 0$ and $x_0 \in \R^{n-1}$, where $h_\alpha$ is given by
\[
h_\alpha(t,x) = \left( 1 + (\alpha - 1) C(t, x) \right)_+^{\frac 1{1-\alpha}}\, .
\]
\end{theorem}

\begin{theorem}
\label{nguyenthmLog}
Let $C$ be an even convex positively $q$-homogeneous function with $q > 1$ such that $C(x) > 0$ for all $x \neq 0$ and let $p > 1$ with $\frac{1}{p} + \frac{1}{q} = 1$. Assume $a \geq 0$. For any smooth function $f$ with compact support on $\R^n$ such that $\|f\|_{L^p(\Rnm, \omega)} = 1$, we have:

\[
Ent_\omega(|f|^p) \leq \frac{n+a}p \log\left[ L_{n,a}(p) \left( \Moma{(K_C)} \right)^{-\frac p{n+a}} \int_\Rnm C^*(\fullgradient f(y)) \omega(y) dy \right]\, ,
\]
where $L_{n,a}(p) = \frac p{n+a} \left(\frac{p-1}{e}\right)^{p-1} \Gamma \left(\frac{n+a+q}{q}\right)^{-\frac{p}{n+a}} $ and

\[
Ent_\omega(|f|^p) := \int_\Rnm |f(y)|^p \log(|f(y)|^p) \omega(y) dy\, .
\]
Moreover, equality holds if

\[
f(t,x) = b e^{-a C(t, x-x_0)}
\]
for some $a > 0$, $b \in \R$ and $x_0 \in \R^{n-1}$ chosen so that $\|f\|_{L^p(\Rnm, \omega)} = 1$.
\end{theorem}

The proofs of Theorems \ref{mainthmSob}, \ref{mainthmGNS} and \ref{mainthmLog} in the case that $p > 1$ are based on two lemmas and use Theorems \ref{nguyenthmSob}, \ref{nguyenthmGNS} and \ref{nguyenthmLog} respectively in a crucial manner, with the appropriate choice of $C$ for each $f$, denoted by $C_f$, as stated in Definition \ref{defi} below.

Before introducing $C_f$, we shall need some notations.

Throughout the remainder of paper we think of $\R^{n-1}$ as the subset  $\{0\} \times \R^{n-1}\subset \R^n$.

For each smooth function $f$ with compact support on $\R^n$, consider

\[
L_f=\{\xi\in\R^{n-1} :\ \|\partialgradient_\xi f\|_{\LW p} \leq 1\}\, ,
\]
which is a convex body in $\R^{n-1}$ defined by the norm

\[
\|\xi\|_f =\left( \int_\Rnm |\partialgradient_\xi f(y)|^p \omega(y) dy \right) ^{\frac1p} =  \|\partialgradient_\xi f\|_\LW{p}\, .
\]

For convenience, we set
\begin{equation}
\label{lem_eqn_Zdef}
\ZZ = \left( \int_\S \| \partialgradient_\xi f\|_\LW{p}^{1-n} d \xi \right) ^{\frac 1 {1 - n}}
\end{equation}
and notice we have the identities
\begin{equation}
\label{lem_eqn_VolZ}
(n-1) \vol(L_f) = Z_p(f)^{1-n}
\end{equation}
and
\begin{equation}
\label{lem_eqn_EZ}
\E_p(f) = c_{n-1,p} Z_p(f)\, .
\end{equation}

We now are ready to introduce the definition of the function $C_f$.
\begin{definition}\label{defi}
Let $1 < p < n+a$ and $(t, x) \in \R^n_+$. For each non-zero smooth function $f$ with compact support on $\R^n$, we define

\[
C_f^*(t,x) :=  \frac 1p \alpha_f |t|^p + \int_\S \|\partialgradient_\xi f\|_\LW{p}^{1-n-p} |\langle x, \xi \rangle|^p d \xi\, ,
\]
where $\alpha_f = p \left( \frac {1 + a} {n-1} \right) \ZZ^{1-n} \TT^{-p}$.

The function $C_f$ is defined as the Legendre transform of $C_f^*$. The specific choice of the constant $\alpha_f$ will be clarified in the last section.
\end{definition}

\begin{lemma}
\label{lem_Cdef}
Let $p$ and $f$ be as in Definition \ref{lem_Cdef}. The function $C_f^*$ is well defined, even, convex, $p$-homogeneous and $C_f^*(y) > 0$ for all $y \neq 0$.
Thus its Legendre transform $C_f$ satisfies the hypotheses of Theorem \ref{nguyenthmSob}, \ref{nguyenthmGNS} and \ref{nguyenthmLog}.
\end{lemma}
\begin{proof}
Let $f$ be a non-zero smooth function with compact support on $\R^n$. We only prove the good definition and strict convexity of $C_f^*$, since the other statements are evident.

First we ensure that $\|\partialgradient_\xi f\|_\LW{p} > 0$ for any $\xi \in \s^{n-2}$. Otherwise, we have $\partialgradient_\xi f(t,x) = 0$ for every $(t,x) \in \R_+^n$ which contradicts the fact that $f$ is non-zero and has compact support on $\R^n$. Then, continuity with respect to $\xi$ gives $\|\partialgradient_\xi f\|_\LW{p} \geq \eps > 0$ for every $\xi \in \s^{n-2}$ which proves the good definition of $C_f^*$.

For the convexity, take $(t_1,x_1), (t_2,x_2) \in \R_+^n$ and $0 \leq \lambda \leq 1$. We have
\begin{eqnarray*}
 C_f^*\left(\lambda (t_1,x_1) + (1-\lambda) (t_2,x_2) \right) &=& \int_{\s^{n-2}}\|\partialgradient_\xi f\|_p^{1-n-p}\left|\lambda {\langle x_1,  \xi\rangle} + (1-\lambda) {\langle x_2,  \xi\rangle} \right|^{p} d\xi\\
 && + \frac 1p \alpha_f |\lambda t_1 + (1-\lambda) t_2|^p\\
 &\leq& \int_{\s^{n-2}}\|\partialgradient_\xi f\|_p^{1-n-p}\left( \lambda |\langle x_1,  \xi\rangle|^p + (1-\lambda) |\langle x_2,  \xi\rangle|^p \right) d\xi\\
 && + \frac 1p \lambda \alpha_f |t_1|^p + \frac 1p (1-\lambda) \alpha_f |t_2|^p\\
 &=& \lambda C_f^*(t_1,x_1) + (1-\lambda) C_f^*(t_2,x_2)\, .
\end{eqnarray*}

\end{proof}

In order to simplify notation, for each $f$ as before, we denote $K_{C_f}$ by $K_f$, $(K_{C_f})_+$ by $(K_f)_+$ and $\|\cdot\|_{K_{C_f}}$ by $\|\cdot\|_{K_f}$ and set

\[
D_f^*(x) = \int_\S \|\partialgradient_\xi f\|_\LW{p}^{1-n-p} |\langle x, \xi \rangle|^p d \xi
\]
for $x \in \R^{n-1}$.

It is easy to see that
\begin{equation}
\label{lem_productLegendre}
C_f(t, x) = \frac {\alpha_f^{1-q}}{q} |t|^q + D_f(x)\, ,
\end{equation}
where $D_f$ is the Legendre transform of $D_f^*$ and $q > 1$ satisfies $\frac{1}{p} + \frac{1}{q} = 1$.

Let $K_{f,t} = \{x\in \R^{n-1}:\ (t, x) \in K_f\}$. Using \eqref{lem_productLegendre} we see that $K_{f,0} = \{x \in \R^{n-1}:\ D_f(x) \leq 1\}$ and
\begin{equation}
\label{lem_Kt_K0_eqn}
K_{f,t} = \left\{x \in \R^{n-1}:\ D_f(x) \leq 1- \frac{\alpha_f^{1-q}}{q} |t|^q \right\} = \left(1-\frac{\alpha_f^{1-q}}{q} |t|^q\right)^{\frac{1}{q}} K_{f,0}\, .
\end{equation}

The next lemma is a central tool in the sequel and states that

\begin{lemma}
\label{lem_K_L}
Let $p$ and $f$ be as in Definition \ref{lem_Cdef} and let $q > 1$ with $\frac{1}{p} + \frac{1}{q} =1$. The relation between $K_{f,0}$ and $L_f$ is given by

\[
K_{f,0} = \left((n+p-1) a_{n-1,p} \vol(L_f)\right)^{\frac{1}{p}} q^{\frac 1q}p^{\frac 1p} \Gamma_p L_f
\]
and as a consequence,

\begin{equation}
\label{lem_K_L_eqn}
\vol(K_{f,0}) = \left(p q^{\frac pq } (n+p-1) a_{n-1,p}\right)^{\frac{n-1}{p}} \vol(L_f)^{\frac{n-1}p} \vol(\Gamma_p L_f)\, ,
\end{equation}
where $a_{n,p}$ is given in Section \ref{sec_notation}.
\end{lemma}
\begin{proof}
See Lemmas $3$ and $4$ of \cite{HJM}.
\end{proof}
%
%

\section{Proof of the main theorems}
\label{sec_proof}

We first prove Theorems \ref{mainthmSob}, \ref{mainthmGNS} and \ref{mainthmLog} for $p > 1$ as corollaries of the following lemma:

\begin{lemma}
\label{main_ineq}
Let $a \geq 0$ and $p, q > 1$ be such that $\frac{1}{p} + \frac{1}{q} =1$. For any smooth function $f$ with compact support on $\R^n$, we have:

\[
\left(\Moma{K_f}\right)^{-\frac p{n+a}} \int_\Rnm C_f^*(\fullgradient f(y)) \omega(y) dy \leq \left( \mathcal R_{n,p,a}  \EE^{\frac{n-1}{n+a}} \TT^{\frac{1 + a}{n+a}}\right)^p\, ,
\]
where $\mathcal R_{n,p,a}$ is computed in the appendix A and is
\begin{align*}
\mathcal R_{n,p,a} &= q^{-\frac1q} \pi ^{-\frac{n-1}{2 (n+a)}} (1+a)^{-\frac{1+a}{p (n+a)}} (n-1)^{-\frac{n-1}{p (n+a)}} \left(\frac{n+a}{p}\right)^{\frac1p} \left(\frac{q \Gamma \left(\frac{n+1}{2}\right) \Gamma \left(\frac{n+a+q}{q}\right)}{\Gamma \left(\frac{1+a}{q}\right) \Gamma \left(\frac{n-1+q}{q}\right)}\right)^{\frac{1}{n+a}}\, .
\end{align*}
\end{lemma}

\begin{proof}
First we compute
\begin{align*}
\int_\Rnm D_f^*(\partialgradient f(y)) \omega(y) dy
&= \int_\Rnm \int_\S \|\partialgradient_\xi f\|_\LW{p}^{1-n-p} |\partialgradient_\xi f(y)|^p d \xi \omega(y) dy\\
&=  \int_\S \|\partialgradient_\xi f\|_\LW{p}^{1-n-p} \int_\Rnm |\partialgradient_\xi f(y)|^p  \omega(y) d y d \xi \\
&= \int_\S \|\partialgradient_\xi f\|_\LW{p}^{1-n} d \xi = \ZZ^{1-n}\, ,\\
\int_\Rnm C^*(\fullgradient f(y)) \omega(y) dy
&= \frac {\alpha_f}{p} \TT^p + \ZZ^{1-n} \\
&= \frac {n+a} {n-1} \ZZ^{1-n}\, .
\end{align*}

By the formula \eqref{lem_Kt_K0_eqn}, we have
\begin{align*}
\int_{(K_{f})_+} y_n^a dy
&= \int_0^1 t^a \vol(K_{f,t})dt\\
&= \int_0^1 t^a \left(1-\frac{\alpha_f^{1-q}}{q} |t|^q\right)^\frac{n-1}{q} \vol(K_{f,0})dt\\
&= q^{\frac{1+a}{q}} \alpha_f^{\frac{1+a}{p}} \vol(K_{f,0}) \int_0^1 s^a \left(1-s^q\right)^{\frac{n-1}{q}} ds\, .
\end{align*}
Then, by Lemma \ref{lem_K_L} and the $\Lp$ Busemann-Petty centroid inequality \eqref{not_BPineq},
\begin{align*}
 \left(\Moma{K_f}\right)^{-\frac p{n+a}}
&=  \left( q^{\frac{1+a}{q}} \alpha_f^{\frac{1+a}{p}} \vol(K_{f,0}) \int_0^1 s^a \left(1-s^q\right)^{\frac{n-1}{q}} ds \right)^{-\frac p {n+a}}\\
&\leq \left(p (n+p-1) a_{n-1,p}\right)^{-\frac{n-1}{n+a}} \alpha _f^{-\frac{1+a}{n+a}} q^{\frac{p- p n- pa+q}{q n+a}} \\\times& \vol(L_f)^{-\frac{n+p-1}{n+a}} \left(\frac{\Gamma \left(\frac{1+a}{q}\right) \Gamma \left(\frac{n+q-1}{q}\right)}{\Gamma \left(\frac{n+a+q}{q}\right)}\right)^{-\frac{p}{n+a}}\, ,
\end{align*}
which yields
\begin{equation}
\label{lastformula_lemma}
\left(\Moma{K_f}\right)^{-\frac p{n+a}} \int_\Rnm C_f^*(\fullgradient f(y)) \omega(y) dy
 = \EE^{\frac{(n-1) p}{n+a}} \TT^{\frac{(1+a) p}{n+a}}
\end{equation}
\[
\times p^{-1} (n+a) {(1+a)^{-\frac{1+a}{n+a}} (n+p-1)^{-\frac{n-1}{n+a}} q^{\frac{p}{n+a}-p+1} a_{n-1,p}^{-\frac{n-1}{n+a}} \left(\frac{(n-1) c_{n-1,p}^{1-n}\; \Gamma \left(\frac{n+a+q}{q}\right)}{\Gamma \left(\frac{1+a}{q}\right) \Gamma \left(\frac{n+q-1}{q}\right)}\right)^{\frac{p}{n+a}}}
\]
\[
= \left( \mathcal R_{n,p,a}  \EE^{\frac{n-1}{n+a}} \TT^{\frac{1+a}{n+a}}\right)^p\, .
\]

\end{proof}

\begin{center}
{\bf Proof of Theorem \ref{mainthmSob} and Corollary \ref{conseqaffSoblev}}
\end{center}

First assume $p > 1$. From Theorem \ref{nguyenthmSob} with $C = C_f$ and Lemma \ref{main_ineq}, we deduce that
\begin{align}
\|f\|_\LW{p^*} 	&\leq S(n,a,p) \mathcal R_{n,p,a}  \EE^{\frac{n-1}{n+a}} \TT^{\frac{1+a}{n+a}} \nonumber \\
		&= \mathcal S_{n,p,a}  \EE^{\frac{n-1}{n+a}} \TT^{\frac{1+a}{n+a}}\, , \label{lastformulasob}
\end{align}
where $\mathcal S_{n,p,a}$ is computed in the appendix A, and so \eqref{mainthmSob_ineq} follows.

Using Young's inequality in the right-hand side above, we get
\begin{align}
\|f\|_\LW{p^*} &\leq \mathcal S_{n,p,a} \left((1+a)^{\frac{1+a}{n+a}} (n-1)^{\frac{n-1}{n+a}}\frac{\EE^p+\TT^p}{n+a}\right)^{\frac 1p} \label{ineqSob} \\
&= \mathcal K_{n,p,a} \left( \EE^p + \TT^p \right)^{\frac 1p}\, , \nonumber
\end{align}
where $\mathcal K_{n,p,a} = \mathcal S_{n,p,a} (1+a)^{\frac{1+a}{p (n+a)}} (n-1)^{\frac{n-1}{p (n+a)}} (n+a)^{-\frac{1}{p}}$.

We prove now that \eqref{mainthmSob_ineq} is sharp.

Consider Theorem \ref{nguyenthmSob} with the function $C(t, x) = \frac{|t|^q}{q} + \frac {|x|^q}{q}$ and its Legendre transform $C^*(t, x) = \frac{|t|^p}{p} + \frac {|x|^p}{p}$. In this case, one obtains a classical inequality whose constant computed by Nguyen \cite{Ng} is exactly $\mathcal K_{n,p,a}$. Note that plugging (\ref{stronger}) in (\ref{ineqSob}), we derive the same sharp inequality, that is

\begin{equation}
\label{prf_nguyen_euclidean_ineq}
\|f\|_\LW{p^*} \leq \mathcal K_{n,p,a} \left( \GG^p + \TT^p \right)^{\frac{1}{p}}\, .
\end{equation}

Using the invariance of \eqref{lastformulasob} with respect to \eqref{mainthmSob_invariantmatrix} and that the extremal functions of \eqref{prf_nguyen_euclidean_ineq} are known of \cite{Ng}, the following are extremal functions of \eqref{lastformulasob}:

\[
f(t,x) = c h_{p,a}( \lambda t, B (x-x_0))\, ,
\]
where

\[
h_{p,a}(t, x) = \left(1 + |t|^p + |x|^p \right)^{-\frac {n + a -p} p }\, .
\]

Let $f$ be an extremal function of \eqref{lastformulasob}. Then the equality case of the $\Lp$ Busemann-Petty centroid inequality implies that $K_{f,0}$ must be an ellipsoid, so that by the definition of $K_{f,0}$, we have $D_f(x) = |B x|^q$ for some invertible matrix $B$. Since we also have equality in \eqref{nguyenthmSob_ineq} with $C = C_f$, we conclude that $f$ must have the form given in Theorem \ref{nguyenthmSob} with $C(t,x) = a |t|^q + |B x|^q$. So, we characterize the extremal functions of \eqref{mainthmSob_ineq} for any $p > 1$. As $p$ approaches $1$, the sharp constant $\mathcal S_{n,p,a}$ tends to a constant $\mathcal S_{n,1,a}$. The sharpness of the latter follows by noting that the inequality \eqref{lastformulasob} becomes equality for $f(x) = {\bf 1}_\B( \lambda t, B(x-x_0))$.

Finally, for the proof of Corollary \ref{conseqaffSoblev}, it suffices to get $\lambda$ so that the terms $\EE$ and $\TT$ are equal for the above extremals. It is proved in the appendix B that the equality holds if, and only if, $\lambda = \det(B)^{\frac 1{n-1}}$. This concludes the section.\ \rule {1.5mm}{1.5mm}\\

\begin{center}
{\bf Proof of Theorem \ref{mainthmGNS}}
\end{center}

Let $p > 1$. We analyze two cases separately.\\

\emph{The case $\alpha > 1$:}

From inequality \eqref{nguyenthmGNS_ineq1}, Theorem \ref{nguyenthmGNS} with $C = C_f$ yields

\[
\|f\|_{\LW{\alpha p}} \leq G_{n,a}(\alpha, p) \left( \Moma{K_f} \right)^{-\frac \theta{n+a}} \left( \int_\Rnm C_f^*(\partialgradient f(y)) \omega(y) dx) \right)^{\frac\theta p} \|f\|^{1-\theta}_{\LW{\alpha(p-1)+1} }\, .
\]

In the same spirit of the previous proof, thanks to Lemma \ref{main_ineq}, we derive

\begin{align}
\|f\|_{\LW{\alpha p}} &\leq \left( G_{n,a}(\alpha, p)^\frac{1}{\theta} \mathcal R_{n,p,a}  \EE^{\frac{n-1}{n + a}} \TT^{\frac{1 + a}{n + a}} \right)^{\theta} \|f\|^{1-\theta}_{\LW{\alpha(p-1)+1} } \nonumber \\
		&= \left( \mathcal G_{n,p,a,\alpha} \EE^{\frac{n-1}{n + a}} \TT^{\frac{1 + a}{n + a}} \right)^{\theta} \|f\|^{1-\theta}_{\LW{\alpha(p-1)+1} }\, , \label{lastformulagn}
\end{align}
where $\mathcal G_{n,p,a,\alpha} $ is computed in the appendix A. Moreover, using the extremals of the classical counterpart, by straightforward computations, it follows that the functions stated in Theorem \ref{mainthmGNS} are extremals of (\ref{lastformulagn}) for any $p > 1$. Thus, the sharp inequality \eqref{mainthmGNS_ineq1} is proved. \\

\emph{The case $\alpha < 1$:}

Similarly, from inequality \eqref{nguyenthmGNS_ineq2}, Theorem \ref{nguyenthmGNS} with $C = C_f$ gives

\[
\|f\|_{\LW{\alpha(p-1)+1} } \leq N_{n,a}(\alpha, p) \left( \Moma{K_f} \right)^{-\frac \theta{n_a}} \left( \int_\Rnm C_f^*(\partialgradient f(y)) \omega(y) dx) \right)^{\frac\theta p} \|f\|^{1-\theta}_{\LW{\alpha p} }\, .
\]

Using Lemma \ref{main_ineq}, we get

\begin{align}
\|f\|_{\LW{\alpha(p-1)+1}} &\leq \left(  N_{n,a}(\alpha, p)^\frac{1}{\theta} \mathcal R_{n,p,a}  \EE^{\frac{n-1}{a+n}} \TT^{\frac{1+ a}{n + a}} \right)^{\theta} \|f\|^{1-\theta}_{\LW{\alpha p} } \nonumber \\
		&= \left( \mathcal N_{n,p,a,\alpha} \EE^{\frac{n-1}{n+a}} \TT^{\frac{1+a}{n+a}} \right)^{\theta} \|f\|^{1-\theta}_{\LW{\alpha p}}\, , \label{lastformulagn1}
\end{align}
where the computation of $\mathcal N_{n,p,a,\alpha}$ and its achievement by extremal are also done in the appendix A.\\

Lastly, letting $p \rightarrow 1^+$, one easily deduces that $\mathcal G_{n,p,a,\alpha} $ and $\mathcal N_{n,p,a,\alpha}$ tends to $\mathcal G_{n,1,a,\alpha} = \mathcal S_{n,1,a} = \mathcal N_{n,1,a,\alpha}$ and one easily sees that $f(x) = {\bf 1}_\B( \lambda t, B(x-x_0))$ is extremal in both limit cases.\ \rule {1.5mm}{1.5mm} \\

\begin{center}
{\bf Proof of Theorem \ref{mainthmLog}}
\end{center}

First assume $p > 1$. By Theorem \ref{nguyenthmLog} with $C=C_f$, we have

\[
\int_\Rnm |f(y)|^p \log(|f(y)|^p) \omega(y) dy \leq \frac{n+a}p \log\left[ L_{n,a}(p) \left( \Moma{K_C} \right)^{-\frac p{n+a}} \int_\Rnm C_f^*(\partialgradient f(y)) \omega(y) dx \right]\, .
\]

So, Lemma \ref{main_ineq} produces

\begin{align}
\int_\Rnm |f(y)|^p \log(|f(y)|^p) \omega(y) dy &\leq \frac{n+a}p \log\left[ L_{n,a}(p)  \left( \mathcal R_{n,p,a}  \EE^{\frac{n-1}{n + a}} \TT^{\frac{1 + a}{n+a}}\right)^p \right] \nonumber \\
		&= \frac{n + a}p \log\left[ \mathcal L_{n,p,a}  \left( \EE^{\frac{n-1}{n+a}} \TT^{\frac{1+a}{n+a}}\right)^p \right]\, . \label{lastformulalog}
\end{align}
The computation of $\mathcal L_{n,p,a}$ is done in the appendix A. The proof that the functions given in the statement are extremal follows ideas previously applied.

Taking the limit $p \rightarrow 1^+$, we have that $\mathcal L_{n,p,a}$ converges also to $\mathcal S_{n,1,a}$ and is direct to verify that functions of the form $f(t,x) = c {\bf 1}_\B( \lambda t, B(x-x_0))$ with $||f||_\LW{1} = 1$ are extremals of (\ref{lastformulalog}) for $p=1$.\ \rule {1.5mm}{1.5mm}

\section{Appendix A}

Here we compute the constants in Theorems \ref{mainthmSob}, \ref{mainthmGNS} and \ref{mainthmLog}.

Let $p > 1$ and $q > 1$ be such that $\frac1p + \frac1q = 1$. We start by computing the constant $\mathcal R_{n,p,a}$ in Lemma \ref{main_ineq}. By formula \eqref{lastformula_lemma}, we have
\begin{align*}
\mathcal R_{n,p,a}
&= (1 + a)^{-\frac{1 + a}{p (n + a)}}  q^{-\frac1q} \left(\frac{(n-1) q \Gamma \left(\frac{n+a+q}{q}\right)}{\Gamma \left(\frac{1+a}{q}\right) \Gamma \left(\frac{n-1+q}{q}\right)}\right)^{\frac{1}{n+a}}\times  \left(\frac{p}{n+a}\right)^{-\frac1p} \left((n+p-1) a_{n-1,p}\right)^{-\frac{n-1}{p (n+a)}} c_{n-1,p}^{-\frac{n-1}{n+a}}\\
&= (1+a)^{-\frac{1+a}{p (n+a)}} (n-1)^{-\frac{n-1}{p (n+a)}} \pi ^{-\frac{n-1}{2 (n+a)}} q^{-\frac1q} \left(\frac{q \Gamma \left(\frac{n+1}{2}\right) \Gamma \left(\frac{n+a+q}{q}\right)}{\Gamma \left(\frac{1+a}{q}\right) \Gamma \left(\frac{n-1+q}{q}\right)}\right)^{\frac{1}{n+a}} \left(\frac{p}{n+a}\right)^{-\frac1p}\\
&= \pi ^{-\frac{n-1}{2 (n+a)}} (1+a)^{-\frac{1+a}{p (n+a)}} (n-1)^{-\frac{n-1}{p (n+a)}} \left(\frac{n+a}{p}\right)^{\frac1p} q^{\frac{1}{n+a}+\frac{1}{p}-1} \left(\frac{\Gamma \left(\frac{1+a}{q}\right) \Gamma \left(\frac{n-1+q}{q}\right)}{\Gamma \left(\frac{n+1}{2}\right) \Gamma \left(\frac{n+a+q}{q}\right)}\right)^{-\frac{1}{n+a}}\\
&= q^{-\frac1q} \pi ^{-\frac{n-1}{2 (n+a)}} (1+a)^{-\frac{1+a}{p (n+a)}} (n-1)^{-\frac{n-1}{p (n+a)}} \left(\frac{n+a}{p}\right)^{\frac1p} \left(\frac{q \Gamma \left(\frac{n+1}{2}\right) \Gamma \left(\frac{n+a+q}{q}\right)}{\Gamma \left(\frac{1+a}{q}\right) \Gamma \left(\frac{n-1+q}{q}\right)}\right)^{\frac{1}{n+a}}\, .
\end{align*}

For the constant in Theorem \ref{mainthmSob}, by formula \eqref{lastformulasob}, we have
\begin{align*}
\mathcal S_{n,p,a}
&= S(n,a,p) \mathcal R_{n,p,a}\\
&= \pi ^{-\frac{n-1}{2 (n+a)}} (1+a)^{-\frac{1+a}{p (n+a)}} (n-1)^{-\frac{n-1}{p (n+a)}} \left(\frac{p-1}{n+a-p}\right)^{\frac1q} \\
&\times \left(\frac{\Gamma \left(\frac{1+a}{q}\right) \Gamma \left(\frac{n-1+q}{q}\right)}{q \Gamma \left(\frac{n+1}{2}\right) \Gamma \left(\frac{n+a+q}{q}\right)}\right)^{-\frac{1}{n+a}} \left(\frac{\Gamma \left(n+a\right)}{\Gamma \left(\frac{n+a}{p}\right) \Gamma \left(\frac{n+a+q}{q}\right)}\right)^{\frac{1}{n+a}}\\
&= \pi ^{-\frac{n-1}{2 (n+a)}} (1+a)^{-\frac{1+a}{p (n+a)}} (n-1)^{-\frac{n-1}{p (n+a)}} \left(\frac{n+a-p}{p-1}\right)^{-\frac1q} \left(\frac{q \Gamma \left(\frac{n+1}{2}\right) \Gamma (n+a)}{\Gamma \left(\frac{1+a}{q}\right) \Gamma \left(\frac{n-1+q}{q}\right) \Gamma \left(\frac{n+a}{p}\right)}\right)^{\frac{1}{n+a}}\, .
\end{align*}

The constants $\mathcal G_{n,p,a,\alpha}, \mathcal N_{n,p,a,\alpha}$ are obtained similarly:

\begin{align*}
\mathcal G_{n,p,a,\alpha}
&= G_{n,a}(\alpha, p)^{\frac 1 \theta} \mathcal R_{n,p,a}\\
&= q^{-\frac1q} \pi ^{-\frac{n-1}{2 (n+a)}} (1+a)^{-\frac{1+a}{p (n+a)}} (n-1)^{-\frac{n-1}{p (n+a)}} \left(\frac{n+a}{p}\right)^{\frac1p} \left(\frac{q \Gamma \left(\frac{n+1}{2}\right) \Gamma \left(\frac{n+a+q}{q}\right)}{\Gamma \left(\frac{1+a}{q}\right) \Gamma \left(\frac{n+q-1}{q}\right)}\right)^{\frac{1}{n+a}}\\
&\times \left(\left(\frac{(\alpha -1) \left(-n-a+\frac{q (\alpha  (p-1)+1)}{\alpha -1}\right)}{q (\alpha  (p-1)+1)}\right)^{\frac{1}{\alpha  p}} \left(\frac{(\alpha -1)^{p-1} (\alpha  (p-1)+1) q^{1-p}}{n+a}\right)^{\frac{\theta}{p}} \right.\\
&\times \left. \left(\frac{\Gamma \left(\frac{(p-1) \alpha +1}{\alpha -1}\right)}{\Gamma \left(\frac{n+a}{q}+1\right) \Gamma \left(\frac{(p-1) \alpha +1}{\alpha -1}-\frac{n+a}{q}\right)}\right)^{\frac{\theta }{n+a}}\right)^{\frac{1}{\theta}}\\
&= q^{-\frac1q} \pi ^{-\frac{n-1}{2 (n+a)}} (1+a)^{-\frac{1+a}{p (n+a)}} (n-1)^{-\frac{n-1}{p (n+a)}} \left(\frac{n+a}{p}\right)^{\frac1p} \left(\frac{(\alpha -1)^{p-1} (\alpha  (p-1)+1) q^{1-p}}{n+a}\right)^{\frac1p}\\
&\times \left(\frac{(\alpha -1) \left(-n-a+\frac{q (\alpha  (p-1)+1)}{\alpha -1}\right)}{q (\alpha  (p-1)+1)}\right)^{\frac{1}{\alpha  \theta  p}} \left(\frac{q \Gamma \left(\frac{n+1}{2}\right) \Gamma \left(\frac{(p-1) \alpha +1}{\alpha -1}\right)}{\Gamma \left(\frac{1+a}{q}\right) \Gamma \left(\frac{n-1+q}{q}\right) \Gamma \left(\frac{(p-1) \alpha +1}{\alpha -1}-\frac{n+a}{q}\right)}\right)^{\frac{1}{n+a}}\\
&= q^{-\frac1q} \pi ^{-\frac{n-1}{2 (n+a)}} (1+a)^{-\frac{1+a}{p (n+a)}} (n-1)^{-\frac{n-1}{p (n+a)}} \left(\frac{(\alpha -1)^{p-1} (\alpha  (p-1)+1) q^{1-p}}{p}\right)^{\frac1p} \\
&\times \left(\frac{(\alpha -1) \left(-n-a+\frac{q (\alpha  (p-1)+1)}{\alpha -1}\right)}{q (\alpha  (p-1)+1)}\right)^{\frac{1}{\alpha  \theta  p}} \left(\frac{q \Gamma \left(\frac{n+1}{2}\right) \Gamma \left(\frac{(p-1) \alpha +1}{\alpha -1}\right)}{\Gamma \left(\frac{1+a}{q}\right) \Gamma \left(\frac{n-1+q}{q}\right) \Gamma \left(\frac{(p-1) \alpha +1}{\alpha -1}-\frac{n+a}{q}\right)}\right)^{\frac{1}{n+a}}\\
&= p^{-\frac1p} q^{-\frac{1}{q}-1} \pi ^{-\frac{n-1}{2 (n+a)}} (\alpha -1)^{\frac1q} (1+a)^{-\frac{1+a}{p (n+a)}} (n-1)^{-\frac{n-1}{p (n+a)}} (\alpha  p+q)^{1/p} \\
&\times \left(\frac{-\alpha  (n+a) + n+a+\alpha  p+q}{\alpha  p+q}\right)^{\frac{1}{\alpha  \theta  p}} \left(\frac{q \Gamma \left(\frac{n+1}{2}\right) \Gamma \left(\frac{p \alpha }{\alpha -1}-1\right)}{\Gamma \left(\frac{1+a}{q}\right) \Gamma \left(\frac{n}{q}+\frac{1}{p}\right) \Gamma \left(\frac{p \alpha }{\alpha -1}-\frac{n+a}{q}-1\right)}\right)^{\frac{1}{n+a}}
\end{align*}
and

\begin{align*}
\mathcal N_{n,p,a,\alpha}
&= N_{n,a}(\alpha, p)^{\frac 1 \theta} \mathcal R_{n,p,a}\\
&= q^{-\frac1q} \pi ^{-\frac{n-1}{2 (n+a)}} (1+a)^{-\frac{1+a}{p (n+a)}} (n-1)^{-\frac{n-1}{p (n+a)}} \left(\frac{n+a}{p}\right)^{\frac1p} \left(\frac{q \Gamma \left(\frac{n+1}{2}\right) \Gamma \left(\frac{n+a+q}{q}\right)}{\Gamma \left(\frac{1+a}{q}\right) \Gamma \left(\frac{n-1+q}{q}\right)}\right)^{\frac{1}{n+a}}\\
&\times \left(\left(\frac{(1-\alpha )^{p-1} (-\alpha +\alpha  p+1) q^{1-p}}{n+a}\right)^{\frac{\theta}{p}} \left(\frac{q (-\alpha +\alpha  p+1)}{(1-\alpha ) \left(n+a+\frac{q (-\alpha +\alpha  p+1)}{1-\alpha }\right)}\right)^{\frac{1-\theta }{\alpha  p}} \right. \\
&\times \left. \left(\frac{\Gamma \left(\frac{n+a}{q}+\frac{p \alpha -\alpha +1}{1-\alpha }+1\right)}{\Gamma \left(\frac{p \alpha -\alpha +1}{1-\alpha }+1\right) \Gamma \left(\frac{n+a}{q}+1\right)}\right)^{\frac{\theta }{n+a}}\right)^{\frac{1}{\theta} }\\
&= q^{-\frac1q} \pi ^{-\frac{n-1}{2 (n+a)}} (1+a)^{-\frac{1+a}{p (n+a)}} (n-1)^{-\frac{n-1}{p (n+a)}} \left(\frac{n+a}{p}\right)^{\frac1p} \left(\frac{(1-\alpha )^{p-1} (-\alpha +\alpha  p+1) q^{1-p}}{n+a}\right)^{\frac1p}\\
&\times \left(\frac{q \Gamma \left(\frac{n+1}{2}\right) \Gamma \left(\frac{n+a+q}{q}\right)}{\Gamma \left(\frac{1+a}{q}\right) \Gamma \left(\frac{n-1+q}{q}\right)}\right)^{\frac{1}{n+a}} \left(\frac{q (-\alpha +\alpha  p+1)}{(1-\alpha ) \left(n+a+\frac{q (-\alpha +\alpha  p+1)}{1-\alpha }\right)}\right)^{\frac{1-\theta }{\alpha  \theta  p}} \left(\frac{\Gamma \left(\frac{n+a}{q}+\frac{p \alpha -\alpha +1}{1-\alpha }+1\right)}{\Gamma \left(\frac{p \alpha -\alpha +1}{1-\alpha }+1\right) \Gamma \left(\frac{n+a}{q}+1\right)}\right)^{\frac{1}{n+a}}\\
&= q^{-\frac1q} \pi ^{-\frac{n-1}{2 (n+a)}} (1+a)^{-\frac{1+a}{p (n+a)}} (n-1)^{-\frac{n-1}{p (n+a)}} \left(\frac{(1-\alpha )^{p-1} (-\alpha +\alpha  p+1) q^{1-p}}{p}\right)^{\frac1p} \\
&\times \left(\frac{q (-\alpha +\alpha  p+1)}{(1-\alpha ) \left(n+a+\frac{q (-\alpha +\alpha  p+1)}{1-\alpha }\right)}\right)^{\frac{1-\theta }{\alpha  \theta  p}} \left(\frac{q \Gamma \left(\frac{n+1}{2}\right) \Gamma \left(\frac{n+a}{q}+\frac{p \alpha -\alpha +1}{1-\alpha }+1\right)}{\Gamma \left(\frac{1+a}{q}\right) \Gamma \left(\frac{n-1+q}{q}\right) \Gamma \left(\frac{p \alpha -\alpha +1}{1-\alpha }+1\right)}\right)^{\frac{1}{n+a}}\\
&= p^{-\frac1p} q^{-\frac1q} \pi ^{-\frac{n-1}{2 (n+a)}} (1+a)^{-\frac{1+a}{p (n+a)}} (n-1)^{-\frac{n-1}{p (n+a)}} (1-\alpha  (1-p))^{\frac{1-\theta }{\alpha  \theta  p}+\frac{1}{p}}\\
&\times \left(\frac{1-\alpha }{q}\right)^{\frac{\theta -1}{\alpha  \theta  p}+\frac{1}{q}} \left(n+a+\frac{\alpha  p+q}{1-\alpha }\right)^{\frac{\theta -1}{\alpha  \theta  p}} \left(\frac{\Gamma \left(\frac{1+a}{q}\right) \Gamma \left(2-\frac{p \alpha }{\alpha -1}\right) \Gamma \left(\frac{n}{q}+\frac{1}{p}\right)}{q \Gamma \left(\frac{n+1}{2}\right) \Gamma \left(-\frac{p \alpha }{\alpha -1}+\frac{n+a}{q}+2\right)}\right)^{-\frac{1}{n+a}}\, .
\end{align*}

Finally, we compute $\mathcal L_{n,p,a}$.
\begin{align*}
\mathcal L_{n,p,a}
&= L_{n,a}(p) \mathcal R_{n,p,a}^p\\
&= \frac 1{n+a} e^{1-p} (p-1)^{p-1} p \Gamma \left(\frac{n+a}{q}+1\right)^{-\frac{p}{n+a}} \\
&\times \left(q^{-\frac1q} \pi ^{-\frac{n-1}{2 (n+a)}} (1+a)^{-\frac{1+a}{p (n+a)}} (n-1)^{-\frac{n-1}{p (n+a)}} \left(\frac{n+a}{p}\right)^{\frac1p} \left(\frac{q \Gamma \left(\frac{n+1}{2}\right) \Gamma \left(\frac{n+a+q}{q}\right)}{\Gamma \left(\frac{1+a}{q}\right) \Gamma \left(\frac{n-1+q}{q}\right)}\right)^{\frac{1}{n+a}}\right)^p\\
&= e^{1-p} (p-1)^{p-1} (1+a)^{-\frac{1+a}{n+a}} (n-1)^{-\frac{n-1}{n+a}} q^{-\frac{p}{q}} \pi ^{-\frac{(n-1) p}{2 (n+a)}} \Gamma \left(\frac{n+a}{q}+1\right)^{-\frac{p}{n+a}} \left(\frac{q \Gamma \left(\frac{n+1}{2}\right) \Gamma \left(\frac{n+a+q}{q}\right)}{\Gamma \left(\frac{1+a}{q}\right) \Gamma \left(\frac{n-1+q}{q}\right)}\right)^{\frac{p}{n+a}}\\
&= e^{1-p} (p-1)^{p-1} (1+a)^{-\frac{1+a}{n+a}} (n-1)^{-\frac{n-1}{n+a}} q^{-\frac{p}{q}} \pi ^{-\frac{(n-1) p}{2 (n+a)}} \left(\frac{q \Gamma \left(\frac{n+1}{2}\right)}{\Gamma \left(\frac{1+a}{q}\right) \Gamma \left(\frac{n-1+q}{q}\right)}\right)^{\frac{p}{n+a}}\\
&= e^{1-p} p^{p-1} (1+a)^{-\frac{1+a}{n+a}} (n-1)^{-\frac{n-1}{n+a}} q^{2-2 p} \pi ^{-\frac{(n-1) p}{2 (n+a)}} \left(\frac{q \Gamma \left(\frac{n+1}{2}\right)}{\Gamma \left(\frac{1+a}{q}\right) \Gamma \left(\frac{n-1+q}{q}\right)}\right)^{\frac{p}{n+a}}\, . \ \rule {1.5mm}{1.5mm}\\
\end{align*}

\section{Appendix B}

In this appendix we summarize how the involved quantities in the proofs of Theorems \ref{mainthmSob}, \ref{mainthmGNS} and \ref{mainthmLog} vary under a linear change of coordinates.

Let $f$ be as in Theorem \ref{mainthmSob} and denote $f_A(x) = f(Ax)$, where $A$ is an invertible matrix of the form \eqref{mainthmSob_invariantmatrix}.

The equality
\[
{\mathcal E_p(f_A, \omega)} = \lambda^{- \frac ap} \det(A)^{- \frac 1p} \det(B)^{\frac 1{n-1}} \EE \\
\]
follows as in \cite{L-Y-Z-1} and is rather standard.

The following computations are trivial
\begin{align*}
{\left\|{\frac{\partial f_A}{\partial t}}\right\|_\LW{p}} &= \lambda^{\frac{p-a}p} \det(A)^{-\frac 1p} \TT\\
\|f\|_\LW{p_a^*}  &= (\lambda^a \det(A))^{-\frac{n+a-p}{(n+a)p}}\|f\|_\LW{p_a^*}  \\
\alpha_{f_A} &= \lambda^{(a+1)\frac{n+p-1}p - p} \det(B)^{\frac{n-1}p} \alpha_f \\
{\|\partialgradient f_A\|^p_\LW{p}} &= \det(A) \lambda^{-a} {\|\partialgradient f\|^p_\LW{p}}\\
L_{f_a} &= ( \det(A) \lambda^a)^{\frac 1p} A^{-1} L_f\\
D^*{f_A}(v) &= \det(A)^{\frac np - 2} D^*_f(A^T.v).
\end{align*}
We deduce the invariance of inequalities \eqref{mainthmSob_ineq}, \eqref{mainthmGNS_ineq1}, \eqref{mainthmGNS_ineq2} and \eqref{mainthmLog_ineq}.
For inequality \eqref{mainthmSob_ineq_2}, we obtain the same factors in ${\mathcal E_p(f_A, \omega)}$ and ${\left\|{\frac{\partial f_A}{\partial t}}\right\|_\LW{p}}$ whenever $\lambda = \det(B)^{\frac 1{n-1}}$, and
\[ {\mathcal E_p(f_A, \omega)}^p + {\left\|{\frac{\partial f_A}{\partial t}}\right\|^p_\LW{p}}
= \lambda^{p-a} \det(A)^{-1} \left(\EE^p + \TT^p \right) \]
which implies the invariance of the inequality.

\smallskip

\noindent {\bf Acknowledgments:}
The first author was partially supported by Fapemig (APQ-01454-15).
The third author was partially supported by CNPq (PQ 306855/2016-0) and Fapemig (APQ 02574-16).

\end{document}